\documentclass[a4paper,11pt]{amsart}
\usepackage{multicol}
\usepackage{amsaddr}
\usepackage{enumerate}
\usepackage{amssymb} \usepackage{amsfonts} \usepackage{amsmath}
\usepackage{amsthm} \usepackage{epsfig} 
\usepackage{amscd} \usepackage[all]{xy}
\usepackage{wrapfig}
\usepackage{amsbsy}
\usepackage{array}
\usepackage{mdwlist}
\usepackage{mathrsfs}
\usepackage{subfig}
\usepackage{sseq}
\usepackage{color}

\usepackage{caption}
\usepackage[utf8x]{inputenc}
\usepackage[T1]{fontenc}
\usepackage[margin=1.5in]{geometry}

\newtheorem{theorem}{Theorem}[section]
\newtheorem{lemma}[theorem]{Lemma}
\newtheorem{prop}[theorem]{Proposition}

\newtheorem{property}[theorem]{Property}
\newtheorem{fact}[theorem]{Fact}

\theoremstyle{definition}
\newtheorem{ex}[theorem]{Example}

\newtheorem{defn}[theorem]{Definition}

\theoremstyle{remark}
\newtheorem{remark}[theorem]{Remark}

\newcommand{\p} {\ensuremath {\mathbb{P}}}
\newcommand{\E} {\ensuremath {\mathbb{E}}}

\newcommand{\N} {\ensuremath {\mathbb{N}}}
\newcommand{\R} {\ensuremath {\mathbb{R}}}

\newcommand{\Z} {\ensuremath {\mathbb{Z}}}
\newcommand{\I} {\ensuremath {\mathbb{I}}}
\newcommand{\F} {\ensuremath {\mathscr{F}}}
\newcommand{\Gi} {\ensuremath {\mathscr{G}}}
\newcommand{\Wh} {\ensuremath {\mathscr{W}}}

\newcommand{\A} {\ensuremath {\mathscr{A}}}
\newcommand{\X} {\ensuremath {\mathscr{X}}}

\newcommand{\B} {\ensuremath {\mathscr{B}}}

\newcommand{\D} {\ensuremath {\mathscr{D}}}
\newcommand{\Si} {\ensuremath {\mathscr{S}}}
\newcommand{\Ci} {\ensuremath {\mathscr{C}}}
\newcommand{\mo} {\ensuremath {\mathscr{P}}}
\newcommand{\No} {\ensuremath {\mathcal{N}}}
\newcommand{\Qi} {\ensuremath {\mathscr{Q}}}

\title[Asymptotic equivalence for additive processes]{Asymptotic equivalence for inhomogeneous jump diffusion processes and white noise.}
\author{Ester ~Mariucci}
\address{Laboratoire Jean Kuntzmann, Grenoble.}
\email{Ester.Mariucci@imag.fr}
\date{\today}
\begin{document}

\begin{abstract} We prove the global asymptotic equivalence between the experiments generated by the discrete (high frequency) or continuous observation of a path of a time inhomogeneous jump-diffusion process and a Gaussian white noise experiment. Here, the parameter of interest is the drift function and the observation time $T$ can be both bounded or diverging. The approximation is given in the sense of the Le Cam $\Delta$-distance, under some smoothness conditions on the unknown drift function. These asymptotic equivalences are established by constructing explicit Markov kernels that can be used to reproduce one experiment from the other. \end{abstract}

\maketitle

\section*{Introduction}
Consider a sequence of one-dimensional time inhomogeneous jump-diffusion processes $\{X_t\}_{t\geq0}$ defined by
\begin{equation}\label{X}
 X_t=\eta+\int_0^tf(s)ds+\int_0^t\sigma_n(s)dW_s +\sum_{i=1}^{N_t}Y_i,\quad t\in [0,T_n],
\end{equation}
where:
\begin{itemize}
 \item $\eta$ is some random initial condition;
 \item $W=\{W_t\}_{t\geq 0}$ is a standard Brownian motion;
 \item $N=\{N_t\}_{t\geq 0}$ is an inhomogeneous Poisson process with intensity function $\lambda(\cdot)$, independent of $W$;
 \item $(Y_i)_{i\geq 1}$ is a sequence of i.i.d. real random variables with distribution $G$, independent of $W$ and $N$;
 \item $\sigma_n^2(\cdot)$ is supposed to be known. Either $T_n\to\infty$ and $\sigma_n(\cdot)=\sigma(\cdot)$ does not depend on $n$ or $T_n\equiv T$ and $\sigma_n(\cdot)=\varepsilon_n\sigma(\cdot)$ with $\varepsilon_n\to 0$ as $n\to\infty$.
 \item $f(\cdot)$ belongs to some non-parametric class $\F$ making its estimation consistent (e.g. if $T_n\to\infty$ one may require $\F$ to consist of a subclass of periodic functions).
 \item $\lambda(\cdot)$ and $G(\cdot)$ are unknown and belong to non-parametric classes $\Lambda$ and $\Gi$, respectively.
\end{itemize}
We observe $\{X_t\}_{t\geq 0}$ at discrete times $0=t_0<t_1<\dots<t_n=T_n$ such that $\Delta_n=\max_{1\leq i\leq n}\big\{|t_{i}-t_{i-1}|\big\}\downarrow 0$ as $n$ goes to infinity. We are interested in estimating the drift function $f(\cdot)$ from the discrete data $(X_{t_i})_{i=0}^n$.
At least two natural questions arise:
\begin{enumerate}
 \item How much information about the parameter $f(\cdot)$ do  we lose by observing $(X_{t_i})_{i=0}^n$ instead of $\{X_t\}_{t\in[0,T_n]}$?
 \item Can we construct an easier (read: mathematically more tractable), but equivalent, model from $(X_{t_i})_{i=0}^n$?
 \end{enumerate}
The aim of this paper is to give an answer to questions (1) and (2) by means of the Le Cam theory of statistical experiments. For the basic concepts and a detailed description of the notion of asymptotic equivalence that we shall adopt, we refer to \cite{LeCam,LC2000}. We recall the relevant definitions and properties in Section \ref{sec:lecam}.

One of the main applications of proving an asymptotic equivalence between two sequences of experiments is that it allows to transfer asymptotic risk bounds for any inference problem from one model to the other, at least for bounded loss functions. In particular, if there is an estimator $\tau_1$ in the statistical model $\mo_1=(\X_1,\A_1,\{P_{1,\theta}:\theta\in\Theta\})$ with risk $\int L(\theta,\tau(x))P_{1,\theta}(dx)$, then, for bounded loss functions $L$, there is an estimator $\tau_2$ in $\mo_2$ such that
$$\sup_\theta \bigg|\int L(\theta,\tau_1(x))P_{1,\theta}(dx)-\int L(\theta,\tau_2(x))P_{2,\theta}(dx)\bigg|\to 0,\quad \textnormal{as } n\to \infty.$$
More generally, asymptotic equivalence allows to transfer minimax rates of convergence, up to some constants.

The first asymptotic equivalence results for non-parametric experiments date to 1996 and are due to Brown and Low \cite{BL} and Nussbaum \cite{N96}. This is the first instance of an abundance of works devoted to establishing asymptotic equivalence results for non-parametric experiments. In particular, asymptotic equivalence theory has been developed for non-parametric regression \cite{BL,regression02,GN2002,ro04,C2007,cregression,R2008,C2009,R2013}, non-parametric density estimation models \cite{N96,cmultinomial,j03,BC04}, generalized linear models \cite{GN}, time series \cite{GN2006,NM}, diffusion models \cite{D,CLN,R2006,rmultidimensionale,R11,C14,esterdiffusion}, GARCH model \cite{B}, functional linear regression \cite{M2011} 
and spectral density estimation \cite{GN2010}. Negative results are somewhat harder to come by; the most notable among them are \cite{sam96,B98,wang02}.

There is however a lack of equivalence results concerning processes with jumps. To our knowledge, this is the first one for what concerns the estimation of a drift function issued from a discretely (high frequency) observed Lévy process. We actually allow it to be inhomogeneous in time, i.e. an additive process. In this setting one should also cite the works \cite{nostro,duval} as they are the only ones we know about treating (pure jumps) Lévy processes. However, they both give asymptotic results for the estimation of the Lévy measure.

The interest in Lévy processes is due to them being a building block for stochastic continuous time models with jumps. Because of that, they are widely used in finance, queueing, telecommunications, extreme value theory, quantum theory or biology. Their stationarity property, however, makes them rather inflexible; as a consequence, in recent years additive processes have been preferred in financial modelling (see \cite{tankov}, Chapter 14). It is therefore in this more general setting that we present our results. 

In order to mathematically reformulate questions (1) and (2), let us denote by $(D,\D)$ the Skorokhod space; define $P_{T_n}^{(f,\sigma_n^2,\lambda G)}$ as the law of $\{X_t\}_{t\in[0,T_n]}$ on $(D,\D)$ and $Q_n^{(f,\sigma_n^2,\lambda G)}$ as the law of the vector $(X_{t_0}, X_{t_1},\dots,X_{t_n})$ on $(\R^{n+1},\B(\R^{n+1}))$.

Consider the parameter set $\Theta=\F$. We allow two more degrees of freedom by considering $\lambda \in \Lambda$ and $G \in \Gi$, although these will not be parameters of interest. Let us then consider the following statistical models:
\begin{align*}
\mo_n&=\big(D,\D,\{P_{T_n}^{(f,\sigma_n^2,\lambda G)}:f \in \F\}\big), \\
\mathscr{Q}_n&=\big(\R^{n+1},\B(\R^{n+1}),\{Q_n^{(f,\sigma_n^2,\lambda G)}: f \in \F\}\big). 
\end{align*}
Finally, let us introduce the Gaussian model that will appear in the statement of our main results. For that, let us denote by $(C,\Ci)$ the space of continuous mappings from $[0,\infty)$ into $\R$ endowed with its standard filtration and, coherently with the previous notation, by $P_{T_n}^{(f,\sigma_n^2,0)}$ the law induced on $(C,\Ci)$ by the stochastic process:
\begin{equation}\label{eq:wn}
dy_t=f(t)dt+\sigma_n(t)dW_t, \quad y_0 = 0,\quad t\in[0,T_n].
\end{equation}
We set:
$$\Wh_n=\big(C,\Ci,(P_{T_n}^{(f,\sigma_n^2,0)}:f\in\F)\big).$$

We have already mentioned that asymptotic equivalences can be used to reduce estimation problems from one model to a simpler ones. This is what happens here, the model associated with the discrete or continuous observation of $\{X_t\}$ as in (1) has been proved to be equivalent to that in (2), which is much better studied. For example, consider the two following situations:
\begin{itemize}
 \item $T_n$ is fixed and $\sigma_n(\cdot)=\varepsilon_n\sigma(\cdot)$ with $\varepsilon_n\to 0,$
 \item $T_n$ goes to infinity and $\sigma_n(\cdot)$ is fixed; in this case, also ask that elements of $\F$ have some periodicity assumption.
\end{itemize}
In both these cases, a consistent estimation of $f \in \F$ is possible. Our equivalence result does not rely on assumptions such as these, but it applies to these cases, as well: Indeed, proving equivalence for a class $\F$ automatically implies that the same equivalence holds true for any subclass of $\F$.

We state here our main result in the case in which $\F$ is a functional class consisting of $\alpha$-Hölder, uniformly bounded functions on $\R$, i.e. there exist $B<\infty$, $M<\infty$ and $\alpha\in(0,1]$ such that
$$|f(x)|\leq B \text{ and } |f(x)-f(y)|\leq M|x-y|^{\alpha},\quad \forall x,y\in\R.$$
For the general statements see Section \ref{sec:mainresults}.
\begin{theorem}\label{teo}
 Suppose that $\F$ is a subclass of $\alpha$-Hölder, uniformly bounded functions on $\R$. Let $\sigma_n(\cdot)=\varepsilon_n \sigma(\cdot)$ be such that $0 < m_\sigma \leq \sigma(\cdot) \leq M_\sigma < \infty$  with derivative $\sigma'(\cdot)$ in $L_{\infty}(\R)$. Suppose either:
 \begin{itemize}
  \item $T_n \equiv T < \infty$, $\varepsilon_n \to 0$ and there exists an $L_2 < \infty$ such that for all $\lambda \in \Lambda$, $\|\lambda\|_{L_2([0,T])} < L_2$,
  \item or $T_n \to \infty$, $\varepsilon_n \equiv 1$ and there exist $L_1<\infty$, $L_2 < \infty$ such that for all $\lambda \in \Lambda$, $\|\lambda\|_{L_1(\R)} < L_1$ and $\|\lambda\|_{L_2(\R)} < L_2$.
 \end{itemize}
 Then
 $$\Delta(\mathscr{Q}_n,\Wh_n) \to 0\quad\text{and} \quad  \Delta(\mo_n,\mathscr{Q}_n)\to 0 \text{ as } n\to\infty,$$
 as soon as one of the following two conditions holds
 \begin{enumerate}
  \item $\Gi$ is a subclass of discrete distributions with support on $\Z$: In this case an upper bound for the rate of convergence is $O\Big(\sqrt{\Delta_n}+T_n\Delta_n^{2\alpha}\varepsilon_n^{-2} + T_n\Delta_n\Big)$.
  \item $\Gi$ is a subclass of absolutely continuous distributions with respect to the Lebesgue measure on $\R$ with uniformly bounded densities on a fixed neighborhood of $0$: In this case an upper bound for the rate of convergence is $O\Big(\sqrt[4]{\Delta_n}+T_n\Delta_n^{2\alpha}\varepsilon_n^{-2} + T_n\Delta_n\Big)$.
 \end{enumerate}
\end{theorem}
The paper is organized as follows. Sections \ref{ap} to \ref{sec:parameterspace} fix assumptions and notation.  The main results, as well as examples, are given in Section \ref{sec:mainresults}. A discussion of the results can be found in Section \ref{sec:discussion}. The proofs are postponed to Section \ref{sec:proofs}. They are obtained as a sequence of results proving different (asymptotic) equivalences. Loosely speaking, we first reduce to having in each interval of the discretization at most one jump (Bernoulli approximation, Section \ref{sec:bernoulli}). Secondly, we filter it out via an explicit Markov kernel, reducing ourselves to treating independent Gaussian variables (Section \ref{sec:Markov}). Finally, we apply an argument similar to that in \cite{BL} (Section \ref{sec:bl}) and collect all the pieces to conclude the proofs in Section \ref{sec:fine}. An appendix collects some proofs of general facts about the Le Cam distance that we use in the rest of the paper.

\section{Assumptions and main results}
\subsection{Additive processes}\label{ap}
Time inhomogeneous jump-diffusion processes are a special case of additive processes. Here we briefly recall definitions and properties of this class of processes. 
\begin{defn}
A stochastic process $\{X_t\}_{t\geq 0}$ on $\R$ defined on a probability space $(\Omega,\A,\p)$ is an \emph{additive process} if the following conditions are satisfied.
\begin{enumerate}
\item $X_0=0$ $\p$-a.s.
\item Independent increments: for any choice of $n\geq 1$ and $0\leq t_0<t_1<\ldots<t_n$, random variables $X_{t_0}$, $X_{t_1}-X_{t_0},\dots ,X_{t_n}-X_{t_{n-1}}$ are independent.
\item There is $\Omega_0\in \A$ with $\p(\Omega_0)=1$ such that, for every $\omega\in \Omega_0$, $X_t(\omega)$ is right-continuous in $t\geq 0$ and has left limits in $t>0$.
\item Stochastic continuity:  $\forall\varepsilon>0, \p(|X_{t+h}-X_t|\geq \varepsilon)\to 0$ as $h\to 0$.
\end{enumerate}
\end{defn}
Thanks to the \emph{Lévy-Khintchine formula} (see \cite{tankov}, Theorem 14.1), the characteristic function of any additive process $X=\{X_t\}_{t\in[0,T]}$ can be expressed, for all $u$ in $\R$, as:
\begin{equation}\label{caratteristica}
\E\big[e^{iuX_t}\big]=\exp\Big(iu\int_0^t f(r)dr-\frac{u^2}{2}\int_0^t \sigma^2(r)dr-\int_{\R}(1-e^{iuy}+iuy\I_{\vert y\vert \leq 1})\nu_t(dy)\Big), 
\end{equation}
where $f(\cdot)$ and $\sigma^2(\cdot)$ belongs to $L_1(\R)$ and $\nu_t$ is a positive measure on $\R$ satisfying
$$\nu_t(\{0\})=0 \textnormal{ and } \int_{\R}(y^2\wedge 1)\nu_t(dy)<\infty, \quad \forall t\in[0,T]. $$
In the sequel we shall refer to $(f(t),\sigma^2(t),\nu_t)_{t\in[0,T]}$ as the local characteristics of the process $X$ and a $\nu_t$ as above will be called a \emph{Lévy measure}, for all $t$.
This data characterizes uniquely the law of the process $X$. In the case where $f(\cdot)$ and $\sigma(\cdot)$ are constant functions and $\nu_t=\nu$ for all $t$, the process $X$ satisfying \eqref{caratteristica} is stationary, and is called a \emph{Lévy process} of characteristic triplet $(f,\sigma^2,\nu)$.

Let $D=D([0,\infty),\R)$ be the space of mappings $\omega$ from $[0,\infty)$ into $\R$ that are right-continuous with left limits. Define the \emph{canonical process} $x:D\to D$ by 
$$\forall \omega\in D,\quad x_t(\omega)=\omega_t,\;\;\forall t\geq 0.$$

Let $\D_t$ and $\D$ be the $\sigma$-algebras generated by $\{x_s:0\leq s\leq t\}$ and $\{x_s:0\leq s<\infty\}$, respectively. Let $X$ be an additive process defined on $(\Omega,\A,\p)$ having local characteristics $(f(t),\sigma^2(t),\nu_t)_{t\in[0,T]}$. It is well known that it induces a probability measure $P^{(f,\sigma^2,\nu)}$ on $(D,\D)$ such that $\{x_t\}$ defined on $\big(D,\D,P^{(f,\sigma^2,\nu)}\big)$ is an additive process identical in law with 
$(\{X_t\},\p)$ (that is the local characteristics of $\{x_t\}$ under $P^{(f,\sigma^2,\nu)}$ is $(f(t),\sigma^2(t),\nu_t)_{t\geq 0}$).

In the sequel we will denote by $\big(\{x_t\},P^{(f,\sigma^2,\nu)}\big)$ such an additive process, stressing the probability measure and by $P_t^{(f,\sigma^2,\nu)}$ for the restriction of $P^{(f,\sigma^2,\nu)}$ to $\D_t$.
 
Further, for every function $\omega$ in $D$, we will denote by $\Delta \omega_r$ its jump at the time $r$ and by $\omega^c$, $\omega^d$ its continuous 
and discontinuous part, respectively:

$$\Delta \omega_r = \omega_r - \lim_{s \uparrow r} \omega_s,\ \omega_t^d=\sum_{r \leq t}\Delta \omega_r, \ \omega_t^c=\omega_t-\omega_t^d.$$
Note that, if $\nu_t=0$ for all $t\geq 0$, then $\big(\{x_t\},P^{(f,\sigma^2,0)}\big)$ is a Gaussian process that can be represented on $(\Omega,\A,\p)$ as
\begin{equation}\label{eq:Xcont}
 X_t=\int_0^t f(s)ds+\int_0^t\sigma(s)dW_s,\quad t\geq 0,
\end{equation}
for some standard Brownian motion $W$ on $(\Omega,\A,\p)$. 

A time inhomogeneous jump-diffusion process as in \eqref{X}, observed until the time $T_n$, is an additive process (apart from the possibly non-zero initial condition) with local characteristics $(f(t),\sigma^2(t),\nu_t)_{t\in[0,T_n]}$, where $\nu_t(\cdot)=\lambda(t)G(\cdot)$. We will write $\big(\{x_t\}, P_{T_n}^{(f,\sigma^2,\lambda G)}\big)$ for such a process. Also observe that $\big(\{x_t^c\}, P_{T_n}^{(f,\sigma^2,\lambda G)}\big)$ has the same law as $\big(\{x_t\}, P_{T_n}^{(f,\sigma^2,0)}\big)$. Moreover, thanks to the independence of the increments, the law of the $i$-th increment of \eqref{X} is the convolution product between the Gaussian law $\No\Big(\int_{t_{i-1}}^{t_i}f(s)ds,\int_{t_{i-1}}^{t_i}\sigma^2(s)ds\Big)$ and the law of the variable $\sum_{j=1}^{P_i}Y_j$, where $P_i$ is Poisson of intensity $\lambda_i=\int_{t_{i-1}}^{t_i}\lambda(s)ds$.
\subsection{Le Cam theory of statistical experiments}\label{sec:lecam}
A \emph{statistical model} is a triplet $\mo_j=(\X_j,\A_j,\{P_{j,\theta}; \theta\in\Theta\})$ where $\{P_{j,\theta}; \theta\in\Theta\}$ 
is a family of probability distributions all defined on the same $\sigma$-field $\A_j$ over the \emph{sample space} $\X_j$ and $\Theta$ is the \emph{parameter space}.
The \emph{deficiency} $\delta(\mo_1,\mo_2)$ of $\mo_1$
with respect to $\mo_2$ quantifies ``how much information we lose'' by using $\mo_1$ instead of $\mo_2$ and is defined as
$\delta(\mo_1,\mo_2)=\inf_K\sup_{\theta\in \Theta}||KP_{1,\theta}-P_{2,\theta}||_{TV},$
 where TV stands for ``total variation'' and the infimum is taken over all ``transitions'' $K$ (see \cite{LeCam}, page 18). The general definition of transition is quite involved but, for our purposes, it is enough to know that Markov kernels are special cases of transitions.
 
The Le Cam $\Delta$-distance is defined as the symetrization of $\delta$ and it defines a pseudometric. When $\Delta(\mo_1,\mo_2)=0$  the two statistical models are said to be equivalent.
Two sequences of statistical models $(\mo_{1}^n)_{n\in\N}$ and $(\mo_{2}^n)_{n\in\N}$ are called \emph{asymptotically equivalent}
if $\Delta(\mo_{1}^n,\mo_{2}^n)$ tends to zero as $n$ goes to infinity. 
There are various techniques to bound the $\Delta$-distance. We report below only the properties that are useful for our purposes. For the proofs see, e.g., \cite{LeCam,strasser} and the Appendix.
\begin{property}\label{delta0}
 Let $\mo_j=(\X,\A,\{P_{j,\theta}; \theta\in\Theta\})$, $j=1,2$, be two statistical models having the same sample space and define 
 $\Delta_0(\mo_1,\mo_2):=\sup_{\theta\in\Theta}\|P_{1,\theta}-P_{2,\theta}\|_{TV}.$
 Then, $\Delta(\mo_1,\mo_2)\leq \Delta_0(\mo_1,\mo_2)$.
\end{property}
In particular, Property \ref{delta0} allows us to bound the $\Delta$-distance between statistical models sharing the same sample space by means of classical bounds for the total variation distance. Classical bounds on the latter will thus prove useful:
\begin{fact}[see \cite{LC69}, p. 35]
 Let $P_1$ and $P_2$ be two probability measures on $\X$, dominated by a common measure $\xi$, with densities $g_{i}=\frac{dP_{i}}{d\xi}$, $i=1,2$. Define
 \begin{align*}
  L_1(P_1,P_2)&=\int_{\X} |g_{1}(x)-g_{2}(x)|\xi(dx), \\
  H(P_1,P_2)&=\bigg(\int_{\X} \Big(\sqrt{g_{1}(x)}-\sqrt{g_{2}(x)}\Big)^2\xi(dx)\bigg)^{1/2}.
 \end{align*}
Then,
\begin{equation} \label{h}
 \frac{H^2(P,Q)}{2}\leq \|P_1-P_2\|_{TV}=\frac{1}{2}L_1(P_1,P_2)\leq H(P_1,P_2).
\end{equation}
\end{fact}
\begin{fact}\label{hp}[see \cite{strasser}, Lemma 2.19]
 Let $P$ and $Q$ be two product measures defined on the same sample space: $P=\otimes_{i=1}^n P_i$, $Q=\otimes_{i=1}^n Q_i$. Then
 \begin{equation}
  H ^2(P,Q)\leq \sum_{i=1}^nH^2(P_i,Q_i).
 \end{equation}
 Using \eqref{h}, it follows that
 \begin{equation*}\label{tvprodotto}
  \|P-Q\|_{TV}\leq\sqrt{\sum_{i=1}^n2\|P_i-Q_i\|_{TV}}.
 \end{equation*}
\end{fact}
Below, we collect some well-known facts that can be used to establish asymptotic equivalences. For the convenience of the reader, their proofs can be found in the Appendix.
\begin{fact}\label{fact:gaussiane}
 Let $Q_1\sim\No(\mu_1,\sigma_1^2)$ and $Q_2\sim\No(\mu_2,\sigma_2^2)$. Then
 $$\|Q_1-Q_2\|_{TV}\leq \sqrt{\bigg(1-\frac{\sigma_1}{\sigma_2}\bigg)^2+\frac{(\mu_1-\mu_2)^2}{2\sigma_2^2}}\leq \sqrt{\bigg(1-\frac{\sigma_1^2}{\sigma_2^2}\bigg)^2+\frac{(\mu_1-\mu_2)^2}{2\sigma_2^2}}.$$
\end{fact}

\begin{fact}\label{fact:processigaussiani}
 Let $m_i(\cdot)$ and $\sigma(\cdot)$ be real functions such that $\int_{\R}\frac{m_i(s)^2}{\sigma(s)^2}ds<\infty$, $i=1,2$, with $\sigma(\cdot)>0$. Then, with the same notation as in Section \ref{ap}:
 $$L_1\big(P_t^{(m_1,\sigma^2,0)},P_t^{(m_2,\sigma^2,0)}\big)=2\Big(1-2\phi\Big(-\frac{D_t}{2}\Big)\Big), \quad \forall t> 0,$$
 where $\phi$ denotes the cumulative distribution function of a Gaussian random variable $\No(0,1)$ and 
 $$D_t^2= \int_0^t\frac{(m_1(s)-m_2(s))^2}{\sigma^2(s)}ds.$$
 In particular, $L_1\big(P_t^{(m_1,\sigma^2,0)},P_t^{(m_2,\sigma^2,0)}\big)=O(D_t)$.
 \end{fact}

\begin{property}\label{fatto3}
 Let $\mo_i=(\X_i,\A_i,\{P_{i,\theta}, \theta\in\Theta\})$, $i=1,2$, be two statistical models. 
Let $S:\X_1\to\X_2$ be a sufficient statistics
such that the distribution of $S$ under $P_{1,\theta}$ is equal to $P_{2,\theta}$. Then $\Delta(\mo_1,\mo_2)=0$. 
\end{property}

\subsection{The parameter space}\label{sec:parameterspace}
We now state the different kinds of assumptions on the non-parametric classes $\F$, $\Lambda$ and $\Gi$ that will show up in the statements of the theorems:
\begin{enumerate}[(F1)]
 \item Every $f\in\F$ is continuous and $\sup_{t\in \R}\{|f(t)|: f\in\F\}\leq B$, for some constant $B$.
\item Defining:
\begin{equation}\label{fn}
\bar{f}_n(t)=\left\{
\begin{array}{ll}
f(t_i)&\textnormal{if} \quad t_{i-1}\leq t<t_i,\quad i=1,\dots,n;\\
f(T_n) & \textnormal{if}\quad t=T_n;
\end{array}\right.
\end{equation}
we have
\begin{equation}\label{ip4}
\lim_{n\to\infty}\sup_{f\in\mathcal{F}}\int_{0}^{T_n}\frac{(f(t)-\bar{f}_n(t))^2}{\sigma_n^2(t)}dt=0.
\end{equation}
\end{enumerate}
\begin{enumerate}[(L1)]
 \item Denoting by $\|\cdot\|_{1}$ the $L_{1}$ norm on $\R$, we require $\sup_{\lambda\in\Lambda}\|\lambda\|_{1}\leq L_1$, for some constant $L_1$.
 \item Denoting by $\|\cdot\|_2$ the $L_2$ norm on $\R$, we ask $\sup_{\lambda\in\Lambda}\|\lambda\|_2^2\leq L_2$, for some constant $L_2$.
\end{enumerate}
\begin{enumerate}[(G1)]
 \item $\Gi$ is a subset of discrete distributions concentrated on $\Z$.
 \item $\Gi$ is a subset of absolutely continuous distributions with respect to Lebesgue, $h = \frac{d G}{d \text{Leb}}$. We ask that there are uniform constants $N_1, N_2 > 0$ such that $h \leq N_2$ Leb-a.e. on $[-\frac{1}{N_1}, \frac{1}{N_1}]$.
\end{enumerate}

\subsection{Main results and examples}\label{sec:mainresults}
Recall that models \eqref{X} and \eqref{eq:wn} depend on diffusion coefficients $\sigma_n(\cdot) = \varepsilon_n \sigma(\cdot)$, where $\varepsilon_n$ is either 1 (if $T_n \to \infty$) or $\varepsilon_n \to 0$ (if $T_n = T$ finite). We will assume that $\sigma(\cdot)$ is absolutely continuous, strictly positive, and its logarithmic derivative is uniformly bounded: There exists a constant $C_1$ such that:
 \begin{equation}\label{eq:lnsigma}
  \Big|\frac{d}{dt}\ln\sigma(t)\Big|\leq C_1, \quad t\in\R.
 \end{equation}
Our main results are then:

\begin{theorem}\label{teo1}
 Suppose that the parameter space $\F$ fulfills the assumptions (F1) and (F2) and let $\sigma(\cdot)$ satisfy \eqref{eq:lnsigma} as above. If, in addition, $\Lambda$ and $\Gi$ satisfy Assumptions (L2) and (G1), respectively, then, for $n$ big enough, we have
\begin{align*}
 &\Delta\big(\mo_n,\mathscr{Q}_n\big)=\Delta\big(\mathscr{Q}_n,\Wh_n\big)\leq O\bigg(\sup_{f\in \F}\int_{0}^{T_n}\frac{(f(t)-\bar{f}_n(t))^2}{\sigma_n^2(t)}dt+T_n\Delta_n+\sqrt{\Delta_n}\bigg).\label{teo1,eq1}
   \end{align*}
 Here, the $O$ depends only on the constants $C_1$ and $L_2$.
\end{theorem}

\begin{theorem}\label{teo2}
 Suppose that the parameter space $\F$ fulfills the assumptions (F1) and (F2) and let $\sigma(\cdot)$ be as above. Suppose also there exist $m_\sigma$, $M_\sigma$ such that $0 < m_\sigma \leq \sigma(\cdot) \leq M_\sigma < \infty$. Let $\beta_i=B(t_i-t_{i-1})+\sqrt{\sigma_i}$. Moreover suppose that $\Lambda$ fulfills Assumptions (L1), (L2) and $\Gi$ fulfills Assumption (G2). Then, for $n$ big enough, we have
\begin{align*}
\Delta& (\mo_n,\mathscr{Q}_n)= \Delta(\Wh_n,\mathscr{Q}_n) \leq O\bigg(\sup_{f\in \F}\int_{0}^{T_n}\frac{(f(t)-\bar{f}_n(t))^2}{\sigma_n^2(t)}dt+T_n\Delta_n+\Delta_n^{\frac{1}{4}}\bigg).
\end{align*}
Here the leading terms in the $O$ depend on $L_1$, $N_2$ and $M_\sigma$ only.
\end{theorem}
As a corollary, when $\F$ consists of uniformly bounded $\alpha$-Hölder functions, one retrieves the rates of convergence stated in Theorem \ref{teo}. We now give some examples of situations where our results can be applied.
\begin{ex}The sum of a diffusion process and an inhomogeneous Poisson process:\label{ex:1}
 This corresponds to setting $Y_1\equiv1$, so that $\Gi$ consists of the only Dirac mass in $1$. Let $\sigma_n(\cdot) = \varepsilon_n \sigma(\cdot)$ satisfy \eqref{eq:lnsigma} as above, and $\Lambda$ satisfy Assumption (L2). If $\F$ is a class of $\alpha$-Hölder, uniformly bounded functions on $\R$ with $\alpha\in(0,1]$, for $n$ big enough,  an application of Theorem \ref{teo1} yields:
 $$\Delta(\mo_n,\mathscr{Q}_n)= \Delta(\mathscr{Q}_n,\Wh_n)=O\Big(\sqrt {\Delta_n}+T_n\Delta_n+T_n\Delta_n^{2\alpha}\varepsilon_n^{-2}\Big).$$
\end{ex}
\begin{ex}Merton model inhomogeneous in time:
 This corresponds to $\Gi$ being a parametric class of Gaussian random distributions $\No(m,\Gamma^2)$, $\Gamma>0$. Suppose that $\sigma(\cdot)$ is as in Example \ref{ex:1} and $\Lambda$ satisfies Assumptions (L1) and (L2). Let $\F$ be a class of $\alpha$-Hölder, uniformly bounded functions on $\R$ with $\alpha\in(0,1]$. Then, for $n$ big enough, an application of Theorem \ref{teo2} yields:
 $$\Delta(\mo_n,\mathscr{Q}_n)= \Delta(\mathscr{Q}_n,\Wh_n)=O\Big(\sqrt[4] {\Delta_n}+T_n\Delta_n+T_n\Delta_n^{2\alpha}\varepsilon_n^{-2}\Big).$$
\end{ex}

\subsection{Discussion}\label{sec:discussion}
\begin{remark}
 Hypotheses (F1), (F2) are modeled on those in \cite{BL}. They are satisfied, for example, by any class $\F$ of uniformly bounded $\alpha$-Hölder functions, with $\alpha$ depending on the asymptotics of the data $\Delta_n$, $T_n$, $\varepsilon_n$, as well as by uniformly bounded Sobolev $W^{\alpha, 2}$ functions. Hypothesis \eqref{eq:lnsigma} on $\sigma^2(\cdot)$ also appears in \cite{BL}. The non-parametric classes $\Lambda$ and $\Gi$ were introduced to stress that the precise parameters $\lambda$, $G$ chosen do not play any role in the proofs.
\end{remark}
 
\begin{remark}
 In the case where $\Gi$ satisfies Assumption (G1) (i.e. the $Y_i$'s are discrete), the Markov kernel $K$ in Lemma \ref{discrete} does not depend on $\sigma(\cdot)$. Hence, combining our Theorem \ref{teo1} with the one by Carter \cite{C2007} one can obtain the same equivalence result when $\sigma(\cdot)$ is an unknown nuisance parameter.
\end{remark}

\begin{remark}\label{rmk:rischio}
 An important advantage of showing the asymptotic equivalence between statistical models is that it allows to transfer statistical inference procedures from one model to the other. This is done in such a way that the asymptotic risk remains the same, at least for bounded loss functions. When the proof of such an equivalence is constructive, one can provide a precise recipe for producing, from a sequence of procedures in one problem, an asymptotically equivalent sequence in the other one. Formally, let us consider two sequences of statistical models $\mo_j^n=(\X_{j,n},\A_{j,n},\{P_{j,n,\theta}; \theta\in\Theta\})$ and a decision or action space $(A,\A)$. Furthermore, for every $n$, let us denote by $\rho_{j,n}$ a possibly \emph{randomized decision procedure} in $\mo_j^n$, i.e. a Markov kernel $\rho_{j,n}:(\X_{j,n},\A_{j,n})\mapsto(A,\A)$ and by $R(\mo_{j,n},\rho_{j,n},L_n,\theta)$ the \emph{risk} in the model $\mo_{j,n}$ with respect to the decision rule $\rho_{j,n}$ and the loss function $L_n$. One says that 
 the sequences of procedures $\rho_{1,n}$ and $\rho_{2,n}$ are \emph{asymptotically equivalent} if for any sequence of bounded loss function $L_n$ one has $\lim_{n\to\infty}\sup_{\theta\in\Theta}|R(\mo_{1,n},\rho_{1,n},L_n,\theta)-R(\mo_{2,n},\rho_{2,n},L_n,\theta)| =0$. 
 
 In this paper there are essentially four statistical models that we prove to be mutually asymptotically equivalent: $\mo_n$, $\Wh_n$, $\mathscr Q_n$ and $\tilde \Qi_n$ (which is associated with the observation of the increments of $(y_t)$ as in \eqref{eq:wn}). The proofs of Theorems \ref{teo1} and \ref{teo2} allow us to use the knowledge of a sequence of procedures in $\mo_n$, $\Wh_n$ or $\tilde\Qi_n$ for producing one in  $\mathscr Q_n$.
 
 For example, suppose that $\Gi$ satisfies Assumption (G1) and let $(\delta_n)$ be a sequence of procedures in  $\tilde \Qi_n$. Define a sequence of procedures in  $\mathscr Q_n$ as:
 \begin{equation*}
  \gamma_n(z_0,\dots,z_n):=
     \delta_n\big(z_1 - z_0-[z_1-z_0],\dots,z_n-z_{n-1}-[z_n-z_{n-1}]\big),\quad z_0,\dots,z_n\in\R,
   \end{equation*}
   where $[z]$ denotes the the closest integer to $z$.
Then $(\gamma_n)$ is asymptotically equivalent to $(\delta_n)$. 

Remark that, up to this point, we did not use the knowledge of $\sigma^2(\cdot)$. In particular, if one disposes of a sequence of estimators of $f(\cdot)$ in $\tilde \Qi_n$ an equivalent one can be deduced in $\mathscr Q_n$ also when $\sigma^2(\cdot)$ is unknown. 
\end{remark}

\section{Proofs}\label{sec:proofs}
\subsection{ Bernoulli approximation}\label{sec:bernoulli}
\begin{lemma}\label{bernoulli}
 Let $(N_i)_{i=1}^n, (P_i)_{i=1}^n$, $(Y_i)_{i=1}^n$ and $(\varepsilon_i)_{i=1}^n$ be samples of, respectively, Gaussian random variables $\No(m_i,\sigma_i^2)$, Poisson random variables $\mo(\lambda_i)$, random variables with common distribution $G$ and Bernoulli random variables of parameters $\alpha_i:=\lambda_i e^{-\lambda_i}$. Let us denote by $Q_{N_i}$ (resp. $Q_{(Y_i,P_i)}$, $Q_{(Y_1,\varepsilon_i)}$) the law of $N_i$ (resp. $\sum_{j=1}^{P_i} Y_j$, $\varepsilon_i Y_1$). Then
 \begin{equation}\label{eq:lambda}
  \|\otimes_{i=1}^n Q_{N_i}*Q_{(Y_i,P_i)}-\otimes_{i=1}^n Q_{N_i}*Q_{(Y_1,\varepsilon_i)}\|_{TV}\leq 2\sqrt{\sum_{i=1}^n\lambda_i^2}
 \end{equation}
 where the symbol $*$ denotes the product convolution between measures.
\end{lemma}
\begin{proof}
Observe that:
 \begin{align*}
  \|Q_{N_i}*Q_{(Y_i,P_i)}-Q_{N_i}*Q_{(Y_1,\varepsilon_i)}\|_{TV}&= \sup_{A\in\B(\R)}\bigg|\sum_{k\geq 0} \p\bigg(N_i+\sum_{j=1}^kY_j\in A\bigg)e^{-\lambda_i}\frac{\lambda_i^k}{k!}\\
  &\phantom{= \sup_{A\in\B(\R)}\bigg|}-(1-\alpha_i)\p(N_i\in A)-\alpha_i\p(N_i+Y_1\in A)\bigg|\\
  &=\sup_{A\in\B(\R)}\bigg|\sum_{k\geq 2} \bigg(\p\Big(N_i+\sum_{j=1}^kY_j\in A\Big)-\p(N_i\in A)\bigg)e^{-\lambda_i}\frac{\lambda_i^k}{k!}\bigg|\\
  &\leq2\sum_{k\geq 2}e^{-\lambda_i}\frac{\lambda_i^k}{k!}\leq 2\lambda_i^2.
 \end{align*}
 We get \eqref{eq:lambda} thanks to Fact \ref{hp}.
\end{proof}
\subsection{Explicit construction of Markov kernels}\label{sec:Markov}
\begin{lemma}\label{discrete}
 Let $(N_i)_{i=1}^n$ and $(\varepsilon_i)_{i=1}^n$ be samples of, respectively, Gaussian random variables $\No(m_i,\sigma_i^2)$ with $|m_i|\leq \frac{1}{3}$ and Bernoulli random variables of parameters $\alpha_i:=\lambda_i e^{-\lambda_i}$, $\lambda_i>0$. Moreover, let $Y_1$ be a discrete random variable taking values in $\Z$ and denote by $Q_{N_i}$ (resp. $Q_{(Y_1,\varepsilon_i)}$) the law of $N_i$ (resp. $\varepsilon_i Y_1$). 
 For all $x$ in $\R$ denote by $[x]$ the nearest integer to $x$ and define the Markov kernel
 $$
 K(x,A)=\I_A(x-[x]),\quad \forall A\in\B(\R).
 $$
 Then
 \begin{equation}\label{eq:bernoulli}
  \big\|\otimes_{i=1}^n K(Q_{N_i}*Q_{(Y_1,\varepsilon_i)})-\otimes_{i=1}^n Q_{N_i}\big\|_{TV}\leq \sqrt{2\sum_{i=1}^n\bigg(\frac{6}{\sigma_i}\varphi\Big(\frac{1}{6\sigma_i}\Big)+4\phi\Big(\frac{-1}{6\sigma_i}\Big)\bigg)}
 \end{equation}
 where $*$ stands for the convolution product, $\phi$ denotes the cumulative distribution of a Gaussian random variable $\No(0,1)$ and $\varphi$ the derivative of $\phi$.
 \end{lemma}
 \begin{proof}
  Denote by $g_i(\cdot)$ the density of $N_i$, by $h(\cdot)$ the density of $Y_1$ with respect to the counting measure and define $G_i(x,k):=(1-\alpha_i)g_i(x)+\alpha_i g_i(x-k)$, $\forall x\in \R$, $\forall k\in \Z$.
  We have, for all $i$:
  \begin{align*}
   \|K(Q_{N_i}*Q_{(Y_1,\varepsilon_i)})- Q_{N_i}\big\|_{TV}&=\sup_{A\in\B(\R)}\bigg|\int\I_A(x-[x])\Big[(1-\alpha_i)g_i(x)+\alpha_i\sum_{k\in \Z}h(k)g_i(x-k)\Big]dx\nonumber\\
                                                      &\phantom{\sup_{A}\bigg|\int} -\int \I_A(x)g_i(x)dx\bigg|\nonumber\\
                                                      &\leq\sup_{A\in\B(\R)}\sum_{k\in\Z}h(k)\bigg|\int\Big(\I_A(x-[x])G_i(x,k)-\I_A(x)g_i(x)\Big)dx\bigg|.
  \end{align*}
Writing $\int\I_A(x-[x])G_i(x,k)dx$ as $\sum_{l\in\Z}\int_{-\frac{1}{2}}^{\frac{1}{2}}\I_A(x)G_i(x+l,k)dx$, one can bound $\Big|\int\big(\I_A(x-[x])G_i(x,k)-\I_A(x)g_i(x)\big)dx\Big|$ by the sum of the following three terms:
\begin{align*}
 I&=\bigg|\int_{-\frac{1}{2}}^{\frac{1}{2}}\I_A(x)\Big[G_i(x,k)+G_i(x+k,k)-g_i(x)\Big]dx\bigg|\\
  &=\bigg|\int_{-\frac{1}{2}}^{\frac{1}{2}}\I_A(x)\big[\alpha_ig_i(x-k)+(1-\alpha_i)g_i(x+k)\big]dx\bigg|\leq\int_{-\frac{1}{2}}^{\frac{1}{2}}\big(g_i(x-k)+g_i(x+k)\big)dx\\
 II&=\sum_{l\in\Z^*-\{k\}}\int_{-\frac{1}{2}}^{\frac{1}{2}}|G_i(x+l,k)|dx\leq \sum_{l\in\Z^*-\{k\}}\int_{-\frac{1}{2}}^{\frac{1}{2}}\big(g_i(x+l)+g_i(x+l-k)\big)dx\\
 III&=\int_{[-\frac{1}{2},\frac{1}{2}]^c}g_i(x)dx.
\end{align*}
Since $\Big|\int\big(\I_A(x-[x])G_i(x,0)-\I_A(x)g_i(x)\big)dx\Big|\leq \int_{[-\frac{1}{2},\frac{1}{2}]^c}g_i(x)dx$ and $h(0)\leq 1$, we obtain
\begin{align*}
 \|K(Q_{N_i}*Q_{(Y_1,\varepsilon_i)})- Q_{N_i}\big\|_{TV}&\leq \sum_{k\in\Z^*}h(k)\int_{-\frac{1}{2}}^{\frac{1}{2}}\big(g_i(x-k)+g_i(x+k)\big)dx\\
   &\ +\sum_{k\in\Z^*, l\in\Z^*-\{k\}}h(k)\int_{-\frac{1}{2}}^{\frac{1}{2}}\big(g_i(x+l)+g_i(x+l-k)\big)dx +2\int_{[-\frac{1}{2},\frac{1}{2}]^c}g_i(x)dx.
\end{align*}
Using the mean value theorem one can write
\begin{align*}
 \int_{-\frac{1}{2}}^{\frac{1}{2}}\big(g_i(x-k)+g_i(x+k)\big)dx&=\phi\Big(\frac{1/2-k-m_i}{\sigma_i}\Big)-\phi\Big(\frac{-1/2-k-m_i}{\sigma_i}\Big)\\
                                                               &\quad +\phi\Big(\frac{1/2+k-m_i}{\sigma_i}\Big)-\phi\Big(\frac{-1/2+k-m_i}{\sigma_i}\Big)\\
                                                               &=\frac{1}{\sigma_i}\big(\varphi(\xi_{1,k})+\varphi(\xi_{2,k})\big)
\end{align*}
for some $\xi_{1,k}\in\Big[\frac{-1/2-k-m_i}{\sigma_i}, \frac{1/2-k-m_i}{\sigma_i}\Big]$ and $\xi_{2,k}\in\Big[\frac{-1/2+k-m_i}{\sigma_i}, \frac{1/2+k-m_i}{\sigma_i}\Big]$.
In particular, since $|m_i|\leq\frac{1}{3}$ one has that $\varphi(\xi_{j,k})\leq \varphi\Big(\frac{1}{6\sigma_i}\Big)$, $j=1,2$, hence
$$\sum_{k\in\Z^*}h(k)\int_{-\frac{1}{2}}^{\frac{1}{2}}\big(g_i(x-k)+g_i(x+k)\big)dx\leq \sum_{k\in\Z^*}\frac{2h(k)}{\sigma_i}\varphi\Big(\frac{1}{6\sigma_i}\Big)\leq \frac{2}{\sigma_i}\varphi\Big(\frac{1}{6\sigma_i}\Big).$$
In the same way one can write 
$$\int_{-\frac{1}{2}}^{\frac{1}{2}}\big(g_i(x+l)+g_i(x+l-k)\big)dx=\frac{1}{\sigma_i}\big(\varphi(\eta_{1,l})+\varphi(\eta_{2,l-k})\big)$$
for some $\eta_{1,l}\in\Big[\frac{-1/2+l-m_i}{\sigma_i}, \frac{1/2+l-m_i}{\sigma_i}\Big]$ and $\eta_{2,l-k}\in\Big[\frac{-1/2+l-k-m_i}{\sigma_i}, \frac{1/2+l-k-m_i}{\sigma_i}\Big]$.
Then:
\begin{align*}
\sum_{k\in\Z^*, l\in\Z^*-\{k\}}h(k)\int_{-\frac{1}{2}}^{\frac{1}{2}}\big(g_i(x+l)+g_i(x+l-k)\big)dx&\leq \sum_{k\in\Z^*, l\in\Z^*-\{k\}}\frac{h(k)}{\sigma_i}\big(\varphi(\eta_{1,l})+\varphi(\eta_{2,l-k})\big)\\
&\leq \sum_{k\in\Z^*, l\in\Z^*-\{k\}}\frac{h(k)}{\sigma_i}\varphi(\eta_{1,l})+\sum_{k,w\in\Z^*}\frac{h(k)}{\sigma_i}\varphi(\eta_{2,w})\\
&\leq\sum_{k\in\Z^*}\frac{h(k)}{\sigma_i}\sum_{l\in\Z^*}\varphi(\eta_{1,l})+\sum_{k\in\Z^*}\frac{h(k)}{\sigma_i}\sum_{w\in\Z^*}\varphi(\eta_{2,w})\\
&\leq \frac{1}{\sigma_i}\sum_{l\in\Z^*}\big(\varphi(\eta_{1,l})+\varphi(\eta_{2,l})\big).
\end{align*}
Now, $|\eta_{i,l}|\geq\frac{|l|-5/6}{\sigma_i}$, $i=1,2$, so
\begin{align*}
 \frac{1}{\sigma_i}\sum_{l\in\Z^*}\big(\varphi(\eta_{1,l})+\varphi(\eta_{2,l})\big)&\leq \frac{4}{\sigma_i}\varphi\Big(\frac{1}{6\sigma_i}\Big)+\frac{1}{\sigma_i}\sum_{|l|\geq 2}\varphi\Big(\frac{|l|-5/6}{\sigma_i}\Big)\\
 &\leq \frac{4}{\sigma_i}\varphi\Big(\frac{1}{6\sigma_i}\Big)+2\int_{\frac{1}{6\sigma_i}}^{\infty}\varphi(x)dx.
 \end{align*}
Finally, $\int_{[-\frac{1}{2},\frac{1}{2}]^c}g_i(x)dx\leq \int_{[-\frac{1}{6\sigma_i},\frac{1}{6\sigma_i}]^c}\varphi(x)dx=2\phi\Big(-\frac{1}{6\sigma_i}\Big)$.
Using Fact \ref{hp}, these computations imply \eqref{eq:bernoulli}.
 \end{proof}
\begin{remark}
 In the case where $Y_1\equiv 1$ (see Example \ref{ex:1}) one can also consider a, maybe, more natural Markov kernel, that is:
 \begin{equation*}
K(x,A)=\I_A(\Psi(x)),\quad \textnormal{with }\ \Psi(x)=\begin{cases} x &\mbox{if } x\leq \frac{1}{2},\\
x-1 &\mbox{otherwise}.         
 \end{cases}
  \end{equation*}
  However, the rate of convergence in \eqref{eq:bernoulli} turns out to be asymptotically the same regardless of the chosen kernel.

\end{remark}

\begin{lemma}\label{continue}
 Let $(N_i)_{i=1}^n$ and $(\varepsilon_i)_{i=1}^n$ be samples of, respectively, Gaussian random variables $\No(m_i,\sigma_i^2)$ with $|m_i|\leq L$ for some constant $L$ and Bernoulli random variables of parameters $\alpha_i:=\lambda_i e^{-\lambda_i}$. Moreover, let $Y_1$ be a random variable with density $h(\cdot)$ with respect to the Lebesgue measure and denote by $Q_{N_i}$ (resp. $Q_{(Y_1,\varepsilon_i)}$) the law of $N_i$ (resp. $\varepsilon_i Y_1$). 
 Fix a $0<\varepsilon<1$ and define, for all $i$, the Markov kernel
 $$
 K_i(x,A)=\begin{cases} \I_A(x) & \mbox{if }x\in B_i:=[-(L+\sigma_i^{1-\varepsilon}),L+\sigma_i^{1-\varepsilon}], \\ \frac{1}{\sqrt{2\pi\sigma_i^2}}\int_Ae^{-\frac{y^2}{2\sigma_i^2}}dy, & \mbox{if }x\in B_i^c.
\end{cases}
 $$
 Then
 \begin{equation*}
  \big\|\otimes_{i=1}^n K_i(Q_{N_i}*Q_{(Y_1,\varepsilon_i)})-\otimes_{i=1}^n Q_{N_i}\big\|_{TV}\leq \sqrt{2\sum_{i=1}^n\bigg(8\phi(-\sigma_i^{-\varepsilon})+\frac{\alpha_i|m_i|}{\sqrt 2\sigma_i}+2\alpha_i\int_{-2\beta_i}^{2\beta_i}h(y)dy\bigg)}
 \end{equation*}
 where $\phi$ denotes the cumulative distribution of a Gaussian random variable $\No(0,1)$ and $\beta_i=L+\sigma_i^{1-\varepsilon}$.
\end{lemma}
\begin{proof}
 The total variation distance between the measures $K_i(Q_{N_i}*Q_{(Y_1,\varepsilon_i)})$ and $Q_{N_i}$ is bounded by the sum of the following two terms:
 \begin{align*}
  I&=\sup_{A\in \B(\R)}|K_i(Q_{N_i}*Q_{(Y_1,\varepsilon_i)})(A\cap B_i)-Q_{N_i}(A\cap B_i)|,\\
  II&=\sup_{A\in \B(\R)}|K_i(Q_{N_i}*Q_{(Y_1,\varepsilon_i)})(A\cap B_i^c)-Q_{N_i}(A\cap B_i^c)|.
 \end{align*}
Denote by $Q_{\tilde N_i}$ the distribution of the Gaussian random variable $\tilde N_i \sim \No(0,\sigma_i^2)$, then
\begin{align*}
 I&=\sup_{A\in \B(\R)}\bigg|\alpha_i\Big(\p(N_i+Y_1\in A\cap B_i)+\p(\tilde N\in A\cap B_i)\p(N_i+Y_1\in B_i^c)-\p(N_i\in A\cap B_i)\Big)\\
 &\phantom{\sup_{A\in \B(\R)}\bigg|\alpha_i}+(1-\alpha_i)\p(\tilde N_i\in A\cap B_i)\p(N_i\in B_i^c) \bigg|\\
 &\leq\sup_{A\in \B(\R)}\alpha_i\bigg(\p\big(N_i+Y_1\in A\cap B_i\big)+\Big|\p(N_i+Y_1\in B_i^c)\Big[\p(\tilde N\in A\cap B_i)-\p(\tilde N\in A\cap B_i)\Big]\Big|\\
  &\phantom{\leq\sup_{A\in \B(\R)}\alpha_i}+\Big|\p(N_i\in A\cap B_i)\Big[\p(N_i+Y_1\in B_i^c)-1\Big]\Big|\bigg)+\p(N_i\in B_i^c)\\
 &\leq \alpha_i\big(2\p(N_i+Y_1\in B_i)+\|Q_{\tilde N_i}-Q_{N_i}\|_{TV}\big)+\p(N_i\in B_i^c)\\
 \textnormal{ and }\\
 II&=\sup_{A\in \B(\R)}\big|\p(\tilde N_i\in  A\cap B_i^c)\p(N_i+\varepsilon_iY_1\in B_i^c)-\p(N_i\in A\cap B_i^c)\big|\\
  &\leq \p(\tilde N_i\in B_i^c)+\p(N_i\in B_i^c).
\end{align*}
Now observe that 
\begin{itemize}
 \item $\p(N_i+Y_1\in B_i)\leq \p(|Y_1|>2\beta_i)\p(|N_i|>\beta_i)+\p(|Y_1|\leq 2\beta_i)\leq \p(N_i\in B_i^c)+\int_{-2\beta_i}^{2\beta_i}h(y)dy$,
 \item $\p(N_i\in B_i^c)=\phi\Big(-\frac{L+\sigma_i^{1-\varepsilon}+m_i}{\sigma_i}\Big)+1-\phi\Big(\frac{L+\sigma_i^{1-\varepsilon}-m_i}{\sigma_i}\Big)\leq \phi(-\sigma_i^{-\varepsilon})+1-\phi(\sigma_i^{-\varepsilon})=2\phi(-\sigma_i^{-\varepsilon})$
 \item $\p(\tilde N_i\in B_i^c)=\phi\Big(-\frac{L+\sigma_i^{1-\varepsilon}}{\sigma_i}\Big)+1-\phi\Big(\frac{L+\sigma_i^{1-\varepsilon}}{\sigma_i}\Big)\leq 2\phi(-\sigma_i^{-\varepsilon})$
\end{itemize}
Combining these bounds with Fact \ref{fact:gaussiane} we get:
$$I+II\leq 8\phi(\sigma_i^{-\varepsilon})+\alpha_i\int_{-2\beta_i}^{2\beta_i}h(y)dy+\alpha_i\frac{|m_i|}{\sqrt 2 \sigma_i}.$$
An application of Fact \ref{hp} allows us to conclude the proof. 

\end{proof}
\subsection{Asymptotic equivalence between discretely and continuously observed Gaussian processes}\label{sec:bl}
Let us denote by $\tilde \Qi_n$ the statistical model associated with the observation of the increments $(y_{t_i}-y_{t_{i-1}})_{i=1}^n$ of $(y_t)$ defined as in \eqref{eq:wn}, then we have:
\begin{prop}\label{prop:gaussiane}
 Suppose that the parameter space $\F$ fulfills Assumption (F2) and let $\sigma(\cdot)>0$ be a given absolutely continuous functions on $\R$ satisfying \eqref{eq:lnsigma}.
 Then, the statistical models $\Wh_n$ and $\tilde \Qi_n$ are asymptotically equivalent as $n$ goes to infinity. An upper bound for the rate of convergence is given by $O\Big(\sup_{f\in\F}\int_0^{T_n}\frac{(f(s)-\bar f_n(s))^2}{\sigma_n^2(s)}ds+T_n\Delta_n\Big)$.
\end{prop}
\begin{proof}
 The proof is based on the same ideas as in \cite{BL}. However, since some modifications are needed, we include a complete proof for the convenience of the reader.
 
STEP 1: We start by considering the statistical model, $\bar\mo_n$, associated with a Gaussian process on $[0,T_n]$ with local characteristic $(\bar f_n(t),\sigma_n^2(t),0)_{t\in[0,T_n]}$ (see \eqref{fn} for the definition of $\bar f_n(\cdot)$). Fact \ref{fact:processigaussiani} guarantees that
$$\Delta(\mo_n,\bar\mo_n)=O\Big(\sup_{f\in\F}\int_0^{T_n}\frac{(f(s)-\bar f_n(s))^2}{\sigma_n^2(s)}ds\Big).$$

 STEP 2: By means of the Fisher factorization theorem, one can easily prove that the statistic defined by
 $$S(\omega)=\bigg(\int_{0}^{t_1}\frac{d\omega_t}{\sigma_n^2(t)},\dots,\int_{t_{n-1}}^{T_n}\frac{d\omega_t}{\sigma_n^2(t)}\bigg)$$ 
 is a sufficient statistic for the family of probabilities $\{P_{T_n}^{(\bar f_n,\sigma_n^2,0)}:f\in \F\}$.
Moreover, the law of $S$ under $P_{T_n}^{(\bar f_n,\sigma_n^2,0)}$ is the law on $\R^n$ of a vector composed by $n$ independent Gaussian random variable $\mu_i:=\No\Big(f(t_i)\int_{t_{i-1}}^{t_i}\frac{dt}{\sigma_n^2(t)},\int_{t_{i-1}}^{t_i}\frac{dt}{\sigma_n^2(t)}\Big)$, $i=1,\dots,n$. Let us denote by $P_{i,f}$ the law on $\R$ of $\mu_i$ and by $\Si_n$ the statistical model associated with the statistic $S$, that is
$$\Si_n=\big\{\R^n,\B(\R^n),(\otimes_{i=1}^n P_{i,f}:f\in\F)\big\}.$$
Then, by using Property \ref{fatto3}, we get $\Delta(\bar\mo_n,\Si_n)=0$.
An application of the mean value theorem yields 
$$\int_{t_{i-1}}^{t_i}\frac{ds}{\sigma_n^2(s)}=\frac{(t_i-t_{i-1})}{\sigma_n^2(\xi_i)}, \quad\textnormal{ for a certain }\xi_i \textnormal{ in }[t_{i-1},t_i].$$
This allows us to pass from the model $\Si_n$ to the equivalent one
$$\tilde\Si_n=\big\{\R^n,\B(\R^n),(\otimes_{i=1}^n \tilde P_{i,f}:f\in\F)\big\},$$
with $\tilde P_{i,f}$ denoting the distribution of a Gaussian random variable $\No(f(t_i)(t_i-t_{i-1}),\sigma_n^2(\xi_i)(t_{i}-t_{i-1}))$.

STEP 3: The last step consists in bounding the $\Delta$-distance between $\tilde\Si_n$ and $\tilde \Qi_n$.

Property \ref{delta0} and Facts \ref{hp}--\ref{fact:gaussiane} yield:
\begin{equation*}
 \Delta(\tilde\Si_n,\tilde \Qi_n)\leq \sup_{f\in \F}\sum_{i=1}^n\Bigg[\bigg(1-\frac{\sigma_n^2(\xi_i)(t_{i}-t_{i-1})}{\int_{t_{i-1}}^{t_i}\sigma_n^2(s)ds}\bigg)^2+\frac{\Big(\int_{t_{i-1}}^{t_i}\Big(f(t_i)-f(s)\Big)ds\Big)^2}{2\int_{t_{i-1}}^{t_i}\sigma_n^2(s)ds}\Bigg].
\end{equation*}
For all $i=1,\dots,n$, let $\eta_i$ and $\gamma_i$ be elements in $[t_{i-1},t_i]$ such that:
\begin{equation*}
 \int_{t_{i-1}}^{t_i}\sigma^2(s)ds=\sigma^2(\eta_i)(t_i-t_{i-1}),\quad \int_{t_{i-1}}^{t_i} f(s)ds=f(\gamma_i)(t_i-t_{i-1}).
\end{equation*}
By means of a Taylor expansion of $\sigma_n(\xi_i)/\sigma_n(\eta_i)$ we obtain
 $$\frac{\sigma_n(\xi_i)}{\sigma_n(\eta_i)}=1+\frac{\sigma_n'(\eta_i)}{\sigma_n(\eta_i)}(\xi_i-\eta_i)+O(\xi_i-\eta_i)^2;$$
 hence, thanks to assumption \eqref{eq:lnsigma}, we have
 $$\bigg|\frac{\sigma_n(\xi_i)}{\sigma_n(\eta_i)}\bigg|\leq 1+ C_1(t_i-t_{i-1})+O\big((t_i-t_{i-1})^2\big).$$
 This means that
 $$\Delta\big(\Si_n,\tilde\Qi_n\big)\leq \sup_{f\in \F}\sum_{i=1}^n\frac{(f(t_i)-f(\gamma_i))^2}{2\sigma_n^2(\eta_i)}(t_{i}-t_{i-1})+O(T_n\Delta_n).$$
 Here, the constant $C_1$ is hidden in the $O$. Observe that $\sum_{i=1}^n\frac{(f(t_i)-f(\gamma_i))^2}{2\sigma_n^2(\eta_i)}(t_{i}-t_{i-1})$ is less than $\int_0^{T_n}\frac{(f(s)-\bar f_n(s))^2}{2\sigma_n^2(s)}ds$. Indeed, on the one hand one can write:
 $$\frac{(f(\xi_i)-f(t_i))^2}{\sigma_n^2(\eta_i)}=\frac{\Big(\int_{t_{i-1}}^{t_i}\big(f(s)-f(t_i)\big)ds\Big)^2}{(t_i-t_{i-1})\int_{t_{i-1}}^{t_i}\sigma_n^2(s)ds},$$
 on the other hand, by means of the Hölder inequality, one has
 $$\Big(\int_{t_{i-1}}^{t_i}\big(f(s)-f(t_i)\big)ds\Big)^2\leq \int_{t_{i-1}}^{t_i}\sigma_n^2(s)ds\int_{t_{i-1}}^{t_i}\frac{\big(f(s)-f(t_i)\big)^2}{\sigma_n^2(s)}ds.$$
 Combining these expressions one finds $\displaystyle \sum_{i=1}^n\frac{(f(t_i)-f(\gamma_i))^2}{2\sigma_n^2(\eta_i)}\leq \sum_{i=1}^n\frac{1}{t_{i}-t_{i-1}}\int_{t_{i-1}}^{t_i}\frac{(f(s)-f(t_i))^2}{2\sigma_n^2(s)}ds$, as claimed.
\end{proof}
\begin{prop}\label{prop:stage}
Suppose that for every $f \in \F$, $\int_0^{T_n}\frac{f^2(s)}{\sigma_n^2(s)}ds<\infty$. Then, the statistical models $\mo_n$ and $\Wh_n$ are equivalent.
\end{prop}
\begin{proof}
 The Girsanov theorem assures that the measure $P^{(0,\sigma_n^2,\lambda G)}$ dominates the measure $P^{(f,\sigma_n^2,\lambda G)}$ and the density is given by
 $$\frac{dP^{(f,\sigma_n^2,\lambda G)}}{dP^{(0,\sigma_n^2,\lambda G)}}(x)=\exp\bigg(\int_0^{T_n}\frac{f(s)}{\sigma_n^2(s)}dx_s^c-\frac{1}{2}\int_0^{T_n}\frac{f^2(s)}{\sigma_n^2(s)}ds\bigg).$$
 We conclude the proof using Fact \ref{fatto3} applied to the statistic $S:\omega\to \omega^c$.
\end{proof}

 \subsection{Proofs of Theorems \ref{teo1} and \ref{teo2}}\label{sec:fine}
In order to prove our results we need to introduce the following notations:
\begin{align*}
 m_i&=\int_{t_{i-1}}^{t_i}f(s)ds,\quad \sigma_i^2=\int_{t_{i-1}}^{t_i}\sigma_n^2(s)ds,\quad
\lambda_i=\int_{t_{i-1}}^{t_i}\lambda(s)ds,\quad \alpha_i=\lambda_ie^{-\lambda_i},\quad i=1,\dots,n.
\end{align*}
 As a preliminary remark observe that the model $\mathscr{Q}_n$ is equivalent to the statistical model that observes the $n$ increments $X_{t_i}-X_{t_{i-1}}$ of \eqref{X}. Let us denote by $\mathscr{\hat Q}_n$ this latter and recall that the law of $X_{t_i}-X_{t_{i-1}}$ is the convolution product between the Gaussian law $\No\big(m_i,\sigma_i^2\big)$ and the law of the variable $\sum_{j=1}^{P_i}Y_j$, where $P_i$ is Poisson with intensity $\lambda_i$.
Regardless of the continuous or discrete nature of $Y_1$, the previous remark and Proposition \ref{prop:stage} allow us to state that 
  $\Delta(\mo_n,\mathscr Q_n)=\Delta(\mo_n,\mathscr{\hat Q}_n)=\Delta(\Wh_n,\mathscr{\hat Q}_n)=\Delta(\mathscr Q_n,\Wh_n)$. Now, to control $\Delta(\mo_n,\mathscr{\hat Q}_n)$ suppose first that $\Gi$ satisfies Assumption (G1). On the one hand, for $n$ big enough, $|m_i|\leq B(t_i-t_{i-1})\leq\frac{1}{3}$, hence we can apply Lemmas \ref{bernoulli}--\ref{discrete} obtaining the bound:
$$\Delta\big(\mathscr{\hat Q}_n, \mathscr{\tilde Q}_n\big)\leq 2\sqrt{\sum_{i=1}^n\lambda_i^2}+\sqrt{2\sum_{i=1}^n\bigg(\frac{6}{\sigma_i}\varphi\Big(\frac{1}{6\sigma_i}\Big)+4\phi\Big(\frac{-1}{6\sigma_i}\Big)\bigg)}.$$
Here we have implicitly used the following fact:

\emph{Let $P_i$ be a probability measure on $(E_i,\mathcal{E}_i)$ and $K_i$ a Markov kernel on $(G_i,\mathcal G_i)$. One can then define a Markov kernel $K$ on $(\prod_{i=1}^n E_i,\otimes_{i=1}^n \mathcal{G}_i)$ such that $K(\otimes_{i=1}^nP_i)=\otimes_{i=1}^nK_iP_i$:
 $$K(x_1,\dots,x_n; A_1\times\dots\times A_n):=\prod_{i=1}^nK_i(x_i,A_i),\quad \forall x_i\in E_i,\ \forall A_i\in \mathcal{G}_i.$$}
Also, observe that 
 $$2\sqrt{\sum_{i=1}^n\lambda_i^2}+\sqrt{2\sum_{i=1}^n\bigg(\frac{6}{\sigma_i}\varphi\Big(\frac{1}{6\sigma_i}\Big)+4\phi\Big(\frac{-1}{6\sigma_i}\Big)\bigg)}=O\big(\sqrt{\Delta_n}\big),$$
 where, in the leading term of the $O$, a constant $L_2$ is hidden.
 On the other hand, thanks to Proposition \ref{prop:gaussiane} we have:
\begin{align*}
\Delta(\mathscr{\tilde Q}_n, \Wh_n)&\leq O\bigg(\sup_{f\in \F}\int_{0}^{T_n}\frac{(f(t)-\bar{f}_n(t))^2}{\sigma_n^2(t)}dt+T_n\Delta_n\bigg).
\end{align*}
We then obtain the inequality stated in Theorem \ref{teo1} by means of the triangular inequality.

In the same way, using Lemmas \ref{bernoulli}, \ref{continue} (taking $\varepsilon=\frac{1}{2}$) and Proposition \ref{prop:gaussiane} one can show the inequality in Theorem \ref{teo2}. Remark that one can actually choose any $0<\varepsilon\leq\frac{1}{2}$. Smaller values of $\varepsilon$ give better bounds for the term involving $\beta_i$ in Theorem \ref{teo2}, but, if $\varepsilon\leq \frac{1}{2}$, under the hypotheses of Theorem \ref{teo}, the leading term is $\sum_i\frac{\alpha_iB(t_i-t_{i_1})}{\sqrt 2\sigma_i}=O\big(\sqrt{\Delta_n}\big)$ (here we hide the constants $L_1$, $B$ and $m_\sigma$). 

 \appendix
\section{Proofs of certain properties of the Le Cam $\Delta$-distance}
\begin{proof}[Proof of Fact \ref{fact:gaussiane}]
 By symmetry, we can suppose $\sigma_1\geq \sigma_2$. Denoting by $g_i$ the density of $Q_i$ with respect to Lebesgue, we have:
 $$\frac{g_1}{g_2}(x)=\frac{\sigma_2}{\sigma_1}\exp\bigg(\frac{(x-\mu_2)^2}{2\sigma_2^2}-\frac{(x-\mu_1)^2}{2\sigma_1^2}\bigg).$$
 Thus the Kullback-Leibler divergence is
 \begin{align*}
  D(Q_1,Q_2)=\int_{\R}g_1(x)\ln\frac{g_1(x)}{g_2(x)}dx&=\ln\frac{\sigma_2}{\sigma_1}+\int_{\R}\bigg(\frac{(x-\mu_2)^2}{2\sigma_2^2}-\frac{(x-\mu_1)^2}{2\sigma_1^2}\bigg)g_1(x)dx\\
                                         &=\ln\frac{\sigma_2}{\sigma_1}+\frac{1}{2}\Big(\frac{\sigma_1^2}{\sigma_2^2}-1\Big)+\frac{(\mu_1-\mu_2)^2}{2\sigma_1^2}.
 \end{align*}
Let $r=\frac{\sigma_1}{\sigma_2}\geq 1$ and observe that
$$-\ln r+\frac{1}{2}(r^2-1)\leq (r-1)^2.$$ 
It is well known (see, e.g. Lemma 2.4 in \cite{tsy}) that the total variation distance is bounded by the square root of the Kullback-Leibler divergence, in this way we obtain:
$$\|Q_1-Q_2\|_{TV}\leq \sqrt{\bigg(1-\frac{\sigma_1}{\sigma_2}\bigg)^2+\frac{(\mu_1-\mu_2)^2}{2\sigma_2^2}}.$$
\end{proof}
\begin{lemma}\label{propL}
Let $g_i$, $i=1,2$ be the density of a Gaussian random variable $\No(\mu_i,\sigma^2)$. Then,
\begin{equation}
L_1(g_1,g_2)=\E\bigg|\exp\bigg(X-\frac{(\mu_2-\mu_1)^2}{2\sigma^2}\bigg)-1\bigg|=2\bigg[1-2\phi\Big(\frac{\mu_2-\mu_1}{2\sigma}\Big)\bigg]
\end{equation}
where $X\sim\mathcal{N}\bigg(0,\frac{(\mu_2-\mu_1)^2}{2\sigma^2}\bigg)$ and $\phi$ is the cumulative distribution function of a Gaussian random variable $\mathcal{N}(0,1)$.
\end{lemma}
\begin{proof}
Without loss of generality let us suppose that $\mu_1 \leq \mu_2$. Then we can write:
\begin{align*}
L_1(g_1,g_2)&=\int_{\R}|g_1(x)-g_2(x)|dx=\int_{-\infty}^{\frac{\mu_1+\mu_2}{2}}(g_1(x)-g_2(x))dx+
                               \int_{\frac{\mu_1+\mu_2}{2}}^{\infty}(g_2(x)-g_1(x))dx.
\end{align*}
Observe that
\begin{align*}
\int_{-\infty}^{\frac{\mu_1+\mu_2}{2}}g_1(x)dx=\p\bigg(\mathcal{N}(\mu_1,\sigma^2)\leq\frac{\mu_1+\mu_2}{2}\bigg)
                               =\p\bigg(\mathcal{N}(0,1)\leq\frac{\mu_2-\mu_1}{2\sigma}\bigg)
                               =\phi\bigg(\frac{\mu_2-\mu_1}{2\sigma}\bigg).
\end{align*}
Similarly one has
\begin{align*}\int_{\R}|g_1(x)-g_2(x)|dx&=\phi\bigg(\frac{\mu_2-\mu_1}{2\sigma}\bigg)-\phi\bigg(\frac{\mu_1-\mu_2}{2\sigma}\bigg)+
\bigg(1-\phi\bigg(\frac{\mu_2-\mu_1}{2\sigma}\bigg)\bigg)-\bigg(1-\phi\bigg(\frac{\mu_1-\mu_2}{2\sigma}\bigg)\bigg)\\
&=2\bigg[\phi\bigg(\frac{\mu_1-\mu_2}{2\sigma}\bigg)-\phi\bigg(\frac{\mu_2-\mu_1}{2\sigma}\bigg)\bigg]=2\bigg[1-2\phi\bigg(\frac{\mu_2-\mu_1}{2\sigma}\bigg)\bigg],
\end{align*}
thus,
\begin{equation*}
L_1(g_1,g_2)=2\bigg[1-2\phi\Big(\frac{\mu_2-\mu_1}{2\sigma}\Big)\bigg].
\end{equation*}
On the other hand we can also express the $L_1$-norm between $g_1$ and $g_2$ as
\begin{align*}
L_1(g_1,g_2)&= \frac{1}{\sqrt{2\pi}\sigma}\int_{\R}\bigg|\exp\bigg(-\frac{(x-\mu_1)^2}{2\sigma^2}\bigg)-\exp\bigg(-\frac{(x-\mu_2)^2}{2\sigma^2}\bigg)\bigg|dx\\
       &=\frac{1}{\sqrt{2\pi}\sigma}\int_{\R}\bigg|\exp\bigg(-\frac{(x-\mu_2)^2-2(\mu_1-\mu_2)(x-\mu_2)+(\mu_1-\mu_2)^2}{2\sigma^2}\bigg)-\exp\bigg(-\frac{(x-\mu_2)^2}{2\sigma^2}\bigg)\bigg|dx\\
&=\frac{1}{\sqrt{2\pi}\sigma}\int_{\R}\bigg|\exp\bigg(\frac{2(\mu_1-\mu_2)(x-\mu_2)-(\mu_1-\mu_2)^2}{2\sigma^2}\bigg)-1\bigg|\exp\bigg(-\frac{(x-\mu_2)^2}{2\sigma^2}\bigg)dx\\
&=\E \bigg|\exp\bigg(\frac{2(\mu_1-\mu_2)(Y-\mu_2)-(\mu_1-\mu_2)^2}{2\sigma^2}\bigg)-1\bigg|=\E \bigg|\exp\bigg(\frac{(\mu_1-\mu_2)Z}{\sigma}-\frac{(\mu_1-\mu_2)^2}{2\sigma^2}\bigg)-1\bigg|\\
&=\E\bigg|\exp\bigg(X-\frac{(\mu_1-\mu_2)^2}{2\sigma^2}\bigg)-1\bigg|,
\end{align*}
where $Y\sim\mathcal{N}(\mu_2,\sigma^2)$ et $Z\sim\mathcal{N}(0,1)$.
\end{proof}
\begin{proof}[Proof of Fact \ref{fact:processigaussiani}]
Thanks to the Girsanov theorem one has that, $\forall \omega\in C$ and $\forall t>0$
$$\frac{dP_t^{(m_i,\sigma^2,0)}}{dP_t^{(0,\sigma^2,0)}}(\omega)=\exp\Big(\int_0^t\frac{m_i(t)}{\sigma^2(t)}d\omega_t-\frac{1}{2}\int_I\frac{m_i^2(t)}{\sigma^2(t)}dt\Big)P^{(0,\sigma^2,0)}(d\omega)$$
In particular, $P_t^{(m_1,\sigma^2,0)}$ is absolutely continuous with respect to $P_t^{(m_2,\sigma^2,0)}$ and the density $g=\frac{\displaystyle dP_t^{(m_1,\sigma^2,0)}}{\displaystyle dP_t^{(m_2,\sigma^2,0)}}$ is given by:
\begin{align}\label{den}
g(\omega)=&\nonumber\exp\Big(\int_0^t\frac{m_1(t)-m_2(t)}{\sigma^2(t)}d\omega_t-\frac{1}{2}\int_0^t\frac{m_1^2(t)-m_2^2(t)}{\sigma^2(t)}dt\Big)\\
       =&\exp\Big(\int_0^t\frac{m_1(t)-m_2(t)}{\sigma^2(t)}(d\omega_t-m_2(t)dt)-\frac{1}{2}\int_0^t\frac{(m_1(t)-m_2(t))^2}{\sigma^2(t)}dt\Big).
\end{align} 
Let us denote by $(Z_t)_{t\geq 0}$ the stochastic process satisfying the following EDS:
$$dZ_t=m_2(t)dt+\sigma(t)dW_t, \quad t\geq 0,$$
with $(W_t)_{t\geq 0}$ a standard Brownian motion.
Then we have:,
\begin{align*}
L_1(P_t^{(m_1,\sigma^2,0)},P_t^{(m_2,\sigma^2,0)})&=\int \big|g(\omega)-1\big|\frac{dP_t^{(m_2,\sigma^2,0)}}{dP_t^{(0,\sigma^2,0)}}(\omega)P^{(0,\sigma^2,0)}(d\omega)\\
&=\E_{\p}\bigg|\exp\bigg(\int\frac{m_1(t)-m_2(t)}{\sigma^2(t)}(dZ_t-m_2(t)dt)
-\frac{1}{2}\int\frac{(m_1(t)-m_2(t))^2}{\sigma^2(t)}dt\bigg)-1\bigg|\\
&=\E_{\p}\bigg|\exp\bigg(\int\frac{(m_1(t)-m_2(t))}{\sigma^2(t)}\sigma(t)dW_t
-\frac{1}{2}\int\frac{(m_1(t)-m_2(t))^2}{\sigma^2(t)}dt\bigg)-1\bigg|.
\end{align*}
Observe that the random variable $\int_0^t\frac{(m_1(s)-m_2(s))}{\sigma(s)}dW_s$ has a Gaussian distribution $\mathcal{N}\displaystyle\Big(0,\int_0^t\frac{(\mu(s)-\nu(s))^2}{\sigma^2(s)}ds\Big)$, thus, by means of Lemma \ref{propL}, we can conclude that
$$L_1\Big(P_t^{(m_1,\sigma^2,0)},P_t^{(m_2,\sigma^2,0)}\Big)=2\bigg[1-2\phi\bigg(\frac{1}{2}\sqrt{\int_0^t\frac{(m_1(s)-m_2(s))^2}{\sigma^2(s)}ds}\bigg)\bigg].$$
\end{proof}

\begin{proof}[Proof of Fact \ref{fatto3}]
 In order to prove that $\delta(\mo_1,\mo_2)=0$ it is enough to consider the Markov kernel $M:(\X_1,\A_1)\to(\X_2,\A_2)$ defined as $M(x,B):=\I_B(S(x))$ $\forall x\in\X_1$ and $\forall B\in\A_2$. Conversely, to show that $\delta(\mo_2,\mo_1)=0$ one can consider the Markov kernel $K:(\X_2,\A_2)\to (\X_1,\A_1)$ defined as 
 $K(y,A)=\E_{P_{2,\theta}}(\I_A|S=y)$, $\forall A\in\A_1.$
 Since $S$ is a sufficient statistics, the Markov kernel $K$ does not depend on $\theta$. Denoting by $S_{\#}P_{1,\theta}$ the distribution of $S$ under $P_{1,\theta}$, one has:
 $$KP_{2,\theta}(A)=\int K(y,A)P_{2,\theta}(dy)=\int \E_{P_{2,\theta}}(\I_A|S=y)S_{\#}P_{1,\theta}(dy)=P_{1,\theta}(A).$$
 
\end{proof}

\subsection*{Acknowledgements}

I would like to thank my Ph.D. supervisor Sana Louhichi for several useful discussions. Special thanks go to the associate editor and the anonoymous referee for a very in-depth reading of the first version of this paper; their comments significantly improved the exposition of the paper.

\bibliographystyle{plain}
\bibliography{refs}

\end{document}